\documentclass[11pt,oneside]{article}
\usepackage{supertabular, setspace,hyperref}
\usepackage[utf8]{inputenc}
\usepackage{amsmath}
\usepackage{amsfonts}
\usepackage{amssymb}
\usepackage{amsthm}
\usepackage{mathtools}
\newtheorem{theorem}{Theorem}[section]

\newtheorem{example}[theorem]{Example}

\newtheorem{remark}{\sc Remark}
\newtheorem{lemma}{\sc Lemma}[section]
\newtheorem{corollary}{\sc Corollary}[section]
\newtheorem{definition}{\sc Definition}[section]

\newcommand{\be}{\begin{eqnarray}}
\newcommand{\ee}{\end{eqnarray}}
\newcommand{\Be}{\begin{eqnarray*}}
	\newcommand{\Ee}{\end{eqnarray*}}
\newcommand{\bee}{\begin{equation}}
\newcommand{\eee}{\end{equation}}
\newcommand{\ba}{\begin{array}}
	\newcommand{\ea}{\end{array}}
\newcommand{\bl}{\begin{lemma}}
	\newcommand{\el}{\end{lemma}}
\newcommand{\bd}{\begin{definition}}
	\newcommand{\ed}{\end{definition}}
\newcommand{\bt}{\begin{theorem}}
	\newcommand{\et}{\end{theorem}}
\newcommand{\bp}{\begin{proof}}
	\newcommand{\ep}{\end{proof}}
\newcommand{\bi}{\begin{itemize}}
	\newcommand{\ei}{\end{itemize}}
\newcommand{\br}{\begin{remark}}
	\newcommand{\er}{\end{remark}}
\newcommand{\bc}{\begin{corollary}}
	\newcommand{\ec}{\end{corollary}}
\newcommand{\bex}{\begin{example}}
	\newcommand{\eex}{\end{example}}
\begin{document}
	\date{}
	\title{\textbf{On the geodesics of homogeneous Finsler spaces with some special $(\alpha, \beta)$-metrics}}
	\maketitle
	\begin{center}
		\author{\textbf{Gauree Shanker and Kirandeep Kaur}}
	\end{center}
	\begin{center}
		Department of Mathematics and Statistics\\
		School of Basic and Applied Sciences\\
		Central University of Punjab, Bathinda, Punjab-151001, India\\
		Email:   gshankar@cup.ac.in, kirandeep.kaur@cup.ac.in
	\end{center}
	\begin{center}
		\textbf{Abstract}
	\end{center}
	\begin{small}
			 In this paper, we study geodesics and geodesic vectors for homogeneous exponential Finsler  space and homogeneous infinite series Finsler space. Further,  we find necessary and sufficient condition for a non-zero vector in these homogeneous spaces to be a geodesic vector. \\
			 \end{small}
	\section{Introduction}
	 According to S. S. Chern (\cite{1996 Chern}), Finsler geometry is just the Riemannian geometry without the quadratic restriction. Finsler generalized Riemann's theory in his doctoral thesis (\cite{Finsler}), but his name was eastablished in differential geometry by Cartan (\cite{Cartan}). In 1972, M. Matsumoto (\cite{M.Mat1972}) has introduced the concept of  $(\alpha, \beta)$-metric in Finsler geometry. A Finsler metric of the form $F= \alpha \phi(s), \ s= \dfrac{\beta}{\alpha}$, where $\alpha= \sqrt{a_{ij}(x)y^iy^j}$ is induced by a Riemannian metric $ \tilde{a}=a_{ij}dx^i \otimes dx^j$ on a connected smooth $n$-manifold $M$ and $\beta= b_i(x) y^i$ is a 1-form on $M$ is called an $(\alpha,\beta)$-metric. There are various applications of $(\alpha,\beta)$-metrics in information geometry (\cite{AmariNaga}),  physics and biology  (\cite{AIM}). Some notable contributions on Finsler spaces with exponential and infinite series $(\alpha, \beta)$-metrics can be seen in  (\cite{SB1}, \cite{SB2}, \cite{GKscur}, \cite{GKflagcur},  \cite{SR}). 
	 
	 Geodesic in a Finsler manifold is the generalization of notion of  a straight line in an Eucldean space.  Geodesic can be viewed as a curve that minimizes the distance between two points on the manifold. A geodesic in a homogeneous Finsler space $(G/H, F)$ is called homogeneous geodesic if it is an orbit of a one-parameter subgroup of $G$. Homogeneous geodesics on homogeneous Riemannian manifolds have been studied by many authors (\cite{Gordon}, \cite{Kostant}, \cite{KowVanh}, \cite{Vinberg}).  The existence of homogeneous geodesics in homogeneous Riemannian manifolds is quite an interesting problem. In  (\cite{Kajzer}), Kajzer  prove that there  exist  atleast one homogeneous geodesic in a Riemannian manifold. Further, in  (\cite{Szenthe2000}, \cite{Szenthe2001}), Szenthe prove that there exist infinitely many homogeneous geodesics through identity in a Riemannian manifold $(G, \alpha)$, where $G$ is compact semi-simple Lie group of rank $\geq$ 2. Later, in (\cite{KowSzen}), Kowalski and Szenthe prove the existence of atleast one homogeneous geodesic through each point for any homogeneous Riemannian manifold. Also, Kowalski and Vlasek (\cite{KowVlasek}) have shown that the result proved in (\cite{KowSzen}) is not true in general. They have given some examples of homogeneous Riemannian manifolds of dimension $\geq$ 4 admitting only one homogeneous geodesic.   There are many applications of  Homogeneous geodesics   in mechanics. Arnold (\cite{Arnold}) study the geodesics of left invariant Riemannian metrics on Lie groups, extending Euler's theory of rigid body motion and called homogeneous geodesic as \textquotedblleft relative equilibria\textquotedblright. The equation of motion of many systems in classical mechanics reduces to the geodesic equation in an appropriate Riemannian manifold. For example, T$\acute{o}$th (\cite{Toth})  study trajectories which are orbits of a one-parameter symmetry group in case of Lagrangian and Hamiltonian systems. Also, Lacomba (\cite{Lacomba})  use homogeneous geodesics in the work of Smale's mechanical systems.\\

	 Latifi (\cite{Latifi2007geo}) extend the concept of homogeneous geodesics in homogeneous Finsler spaces. In   (\cite{Latifi2007geo}), he has given a criterion for  charaterization of  geodesic vectors.  Latifi and Razavi (\cite{LatiRaza}) study homogeneous geodesics in a 3-dimensional connected Lie group with a left invariant Randers metric and show that all the geodesics on spaces equipped with such metrics are homogeneous. Habibi, Latifi and Toomanian (\cite{HabLatiToom}) have extended Szenthe's result of homogeneous  geodesics for invariant Finsler metrics. Yan and Deng (\cite{YanDeng2016})  generalize this result and prove that there exists atleast one homogeneous geodesic through each point for any compact homogeneous Finsler space and also extend Kowalski and Szenthe's result to the Randers space. Yan (\cite{Yan2017}) prove the existence of atleast one homogeneous geodesic through each point for a homogeneous Finsler space of odd dimension.  Hosseini and Moghaddam (\cite{HossMoogh}) study the existence of  homogeneous  geodesic in homogeneous $(\alpha, \beta)$-spaces. Du$\check{s}$ek (\cite{Dusek}) indicates a gap in the proof of main result of  (\cite{Yan2017}) and reproves  this result and  also studies homogeneous geodesics on a homogeneous Berwald space or homogeneous reversible Finsler space.  Recently in 2018, Yan and Huang (\cite{YanHuang}) prove that any homogeneous Finsler space admits atleast one homogeneous geodesic through each point.\\

	   
	
	\section{Preliminaries}
	\begin{definition}
		Let $V$ be an n-dimensional real vector space. It is called a Minkowski space
		if there exists a real valued function $F:V \longrightarrow \mathbb{R}$ satisfying the following conditions: 
		\begin{enumerate}
			\item[\bf(a)]  $F$ is smooth on $V \backslash \{0\},$ 
			\item[\bf(b)] $F(v) \geq 0  \ \ \forall \ v \in V,$
			\item[\bf(c)] $F$ is positively homogeneous, i.e., $ F(\lambda v)= \lambda F(v), \ \ \forall \ \lambda > 0, $
			\item[\bf(d)] For a basis $\{v_1,\ v_2, \,..., \ v_n\}$ of $V$ and $y= y^iv_i \in V$, the Hessian matrix $\left( g_{_{ij}}\right)= \left( \dfrac{1}{2} F^2_{y^i y^j} \right)  $ is positive-definite at every point of $V \backslash \{0\}.$
		\end{enumerate} 
	In this case,	$F$ is called a Minkowski norm.
	\end{definition}
	\begin{definition}	
		Let $M$ be a connected smooth manifold.  If there exists a function $F\colon TM \longrightarrow [0, \infty)$ such that $F$ is smooth on the slit tangent bundle $TM \backslash \{0\}$ and the restriction of $F$ to any $T_p(M), \ p \in M$, is a Minkowski norm, then $(M, F)$ is called a \textbf{Finsler space} and  $F$ is called a Finsler metric. 
	\end{definition}
	Let $(M, F)$ be a Finsler space and let $(x^i,y^i)$ be a standard coordinate system in $T_x(M)$. The induced inner product $g_y$ on $T_x(M)$ is given by $g_y(u,v)=g_{ij}(x,y)u^i v^j$, where $u=u^i \dfrac{\partial}{\partial x^i}, \ v=v^i\dfrac{\partial}{\partial x^i} \in T_x(M) $. Also note that $F(x,y)= \sqrt{g_y(y,y)}.$  \\
	
	Shen (\cite{CSBOOK}) has given the  condition for an $(\alpha,\beta)$-metric to be a Finsler metric  in following lemma:
	\begin{lemma} {\label{existence of metric}}
		Let $F=\alpha \phi(s), \ s=\beta/ \alpha,$ where $\alpha$ is a Riemannian metric and $\beta$ is a 1-form whose length with respect to $\alpha$ is bounded above, i.e., $b:=\lVert \beta \rVert_{\alpha} < b_0,$ where $b_0$ is a positive real number. Then $F$ is a Finsler metric if and only if the function $\phi=\phi(s)$ is a smooth positive function on $\left( -b_0, b_0\right) $ and satisfies the following condition:
		$$ \phi(s)-s\phi'(s)+\left( b^2-s^2\right) \phi''(s)>0, \ \ \lvert s\rvert \leq b < b_0.$$
		
	\end{lemma}
	
\begin{definition}
	 A diffeomorphism $\phi\colon  M \longrightarrow M$ on a Finsler space $ (M, F) $ is called an \textbf{isometry} if $F(p,X)= F\left( \phi(p),  d\phi_p(X)\right)$, for any $p \in M$ and $X \in T_p(M).$
\end{definition}

\begin{definition}
	A Lie group $ G $ is called a Lie transformation group of a smooth manifold $ M $ if it has a smooth action on $ M $.
\end{definition}
	\begin{definition}
	Let	$ (M, F)$  be a connected Finsler space. If the action of the group of isometries of $(M,F)$, denoted by $ I(M, F)$, is transitive on $M$, then it is called a homogeneous Finsler space.
	\end{definition}

 Let $(M, Q)$ be a Riemannian manifold. A geodesic $\sigma \colon \mathbb{R} \longrightarrow M$ is called homogeneous geodesic if there exists a one-parameter group of isometries $\phi \colon \mathbb{R} \times M  \longrightarrow M $ such that $\sigma (t)= \phi(t, \sigma(0)), \ t \in \mathbb{R}$. If all the geodesics of a Riemannian manifold are homogeneous, then it is called a g.o. space (geodesic orbit space). Every naturally reductive Riemannian manifold is a g.o. space. The first example of a g.o. space which is not naturally reductive was given by Kaplan (\cite{Kaplan}).\\

In 2014, Yan and Deng (\cite{YanDeng}) have studied Finsler g.o. spaces defined as follows:\\

A Finsler  space $(M, F)$ is called a Finsler g.o. space if every geodesic of $(M, F)$ is the orbit of a one-parameter subgroup of $G= I(M, F)$, i.e., if $\sigma \colon \mathbb{R} \longrightarrow M$ is a geodesic, then $\exists \ \text{a non-zero vector} \ Z \in \mathfrak{g}= Lie(G)$ and $p \in M$ such that $\sigma(t)= exp(tZ).p$.  A Finsler g.o. space has vanishing S-curvature. Also, note that every Finsler g.o. space is homogeneous.\\

A homogeneous Finsler space is a g.o. space if and only if the projections of all the geodesic vectors cover the set $T_{eH}(G/H)- \{0\}.$
\begin{definition}
	Let $(G/H, F)$ be a homogeneous Finsler space and $e$ be the identity of $G$. A geodesic $\sigma(t)$ through the origin $eH$ of $G/H$ is called homogeneous if it is an orbit of a one-parameter subgroup of $G$, i.e., there exists a non-zero vector $X \in \mathfrak{g}= Lie(G)$ such that $\sigma(t)= exp(tX).eH, \ t \in \mathbb{R}$.
\end{definition}

\section{Geodesic vector} 
The problem of studying homogeneous geodesics of a homogeneous space is basically the study of its geodesic vectors.\\

Let $ (M,F)$ be a  homogeneous Finsler space. Then, $M$ can be written as a coset space $G/H$, where $G= I(M,F)$ is a Lie transformation group of $M$ and $H$,  the compact isotropy subgroup of $ I(M,F)$ at some point $x \in M $(\cite{DH}) . Let $ \mathfrak g$ and $ \mathfrak h $ be the Lie algebras of the Lie groups $G$ and $H$ respectively. Also, let $\mathfrak m$ be a subspace of $\mathfrak g$   such that $\text{Ad}(h) \mathfrak m \subset \mathfrak m \ \ \forall \  h \in H,$ and  $ \mathfrak g = \mathfrak h + \mathfrak m $ be a reductive decomposition of $\mathfrak g$. \\

Observe that for any $Y \in \mathfrak g$, the vector field $ Y^*= \dfrac{d}{dt}\left( \exp(tY)H\right) \bigg|_{t=0}  $ is called the fundamental Killing vector field on $G/H$ generated by $Y$ (\cite{KN}).
The canonical projection $ \pi \colon G \longrightarrow G/H$ induces an isomorphism between the subspace $\mathfrak{m}$ and the tangent space $T_{eH}\left( G/H\right) $  through the following map:
$$
\mathfrak m \longrightarrow T_{eH}\left( G/H\right) $$
$$ \ \ \ \ \ \ \ \ \ \ v \longrightarrow \dfrac{d}{dt}\left( \exp(tv)H\right) \rvert_{t=0}. $$
We have $d\pi(Y_{\mathfrak{m}})= Y^{*}_{eH}$. Using the natural identification and scalar product $g _{_{Y^{*}_{eH}}}$ on $T_{eH}\left( G/H\right)$, we get a scalar product $g_{_{Y_{\mathfrak{m}}}}$ on $\mathfrak{m}.$

\begin{definition}
	Let $(G/H, F)$ be a homogeneous Finsler space and $e$ be the identity of $G$. A non-zero vector $X \in \mathfrak{g}$ is called a geodesic vector if the curve $exp(tX).eH$ is a geodesic of $(G/H, F)$.
\end{definition}

The folowing result proved in (\cite{Latifi2007geo}) gives a criterion for a non-zero vector to be a geodesic vector in a homogeneous Finsler space.
\begin{lemma}{\label{Latifigeolemma}}
	A non-zero vector $Y \in \mathfrak{g}$ is a geodesic vector if and only if $$g_{_{Y_{\mathfrak{m}}}}\left( Y_{\mathfrak{m}}, \left[ Y, Z \right]_{\mathfrak{m}} \right)= 0 \ \forall \ Z \in \mathfrak{g}. $$ 
	
\end{lemma} 

Next, we deduce  necessary and sufficient condition for a nonzero vector in a homogeneous Finsler  space with  infinite series $(\alpha, \beta)$-metric to
be a geodesic vector.
\begin{theorem}
	Let $G/H$ be a homogeneous Finsler space with infinite series metric $F=\dfrac{\beta^2}{\beta-\alpha}$ given by an invariant Riemannian metric $\tilde{a}$ and an invariant vector field $\tilde{X}$ such that $\tilde{X}(H)= X.$ Then, a non-zero vector $Y \in \mathfrak{g}$ is a geodesic vector if and only if 
	\begin{equation}{\label{geoinfeq2}}
	\begin{split}
	&\dfrac{\left\langle X, Y_{\mathfrak{m}}\right\rangle^3}{\left( \left\langle  X, Y_{\mathfrak{m}}\right\rangle - \sqrt{ \left\langle Y_{\mathfrak{m}} , Y_{\mathfrak{m}}\right\rangle } \right) ^4 } \times\\ 
	&\bigg[ \left\langle X, \left[ Y, Z\right]_{\mathfrak{m}}\right\rangle \Big\{\left\langle X, Y_{\mathfrak{m}}\right\rangle^2 -4 \left\langle Y_{\mathfrak{m}}, Y_{\mathfrak{m}}\right\rangle ^{3/2} + \left\langle X, Y_{\mathfrak{m}}\right\rangle \sqrt{\left\langle Y_{\mathfrak{m}}, Y_{\mathfrak{m}}\right\rangle} + 2 \left\langle Y_{\mathfrak{m}}, Y_{\mathfrak{m}} \right\rangle    \Big\}\\
	&  +\left\langle Y_{\mathfrak{m}}, \left[ Y, Z\right]_{\mathfrak{m}}\right\rangle \Big\{ \dfrac{\left\langle X, Y_{\mathfrak{m}}\right\rangle^2 }{\sqrt{\left\langle Y_{\mathfrak{m}}, Y_{\mathfrak{m}} \right\rangle }} - \left\langle X, Y_{\mathfrak{m}}\right\rangle \Big\}\bigg] =0.
	\end{split}
	\end{equation}
\end{theorem}
\begin{proof}
	Using lemma 3.1 of (\cite{GKscur}), we can write 
	$$ F\left( Y\right) =  \dfrac{\left\langle  X, Y\right\rangle^2}{\left\langle  X, Y\right\rangle - \sqrt{ \left\langle Y , Y\right\rangle } }. $$
	Also, we know that
	$$ g_{_{Y}}(U, V)= \dfrac{1}{2}\dfrac{\partial^2}{\partial s \partial t}F^2 (Y+sU + tV)\bigg|_{s=t=0}.  $$
	After  some calculations, we get
	\begin{equation}{\label{geoinfeq1}}
	\begin{split}
	g_{_{Y}}(U, V)&= \dfrac{\left\langle X, Y\right\rangle^2}{\left( \left\langle  X, Y\right\rangle - \sqrt{ \left\langle Y , Y\right\rangle } \right) ^4 } \bigg[ \left\langle X, Y\right\rangle^2 \left\langle X, V\right\rangle \left\langle X, U \right\rangle 
	-4 \left\langle Y, Y\right\rangle^ {3/2} \left\langle X, V \right\rangle \left\langle X, U\right\rangle \\
	&\ \ \ \ \ \ + 6\left\langle Y, Y\right\rangle\left\langle X, V\right\rangle \left\langle X, U\right\rangle + \dfrac{\left\langle X, Y \right\rangle^2 \left\langle X, V \right\rangle \left\langle U, Y\right\rangle }{\sqrt{\left\langle Y, Y\right\rangle}} -4 \left\langle X, Y \right\rangle  \left\langle X, V\right\rangle \left\langle U, Y\right\rangle   \\
	& \ \ \ \ \ \ - \dfrac{\left\langle X, Y\right\rangle^3 \left\langle U, Y \right\rangle \left\langle V, Y\right\rangle}{\left\langle Y, Y\right\rangle ^{3/2}} + \dfrac{\left\langle X, Y\right\rangle ^3 \left\langle U, V\right\rangle }{\sqrt{\left\langle Y, Y\right\rangle}} +  \dfrac{ 4 \left\langle X, Y\right\rangle ^2 \left\langle U, Y\right\rangle  \left\langle V, Y\right\rangle  }{\left\langle Y, Y\right\rangle } \\
	& \ \ \ \ \ \ - \left\langle X, Y \right\rangle^2 \left\langle U, V\right\rangle + \dfrac{\left\langle X, Y \right\rangle^2 \left\langle X, U\right\rangle \left\langle V, Y \right\rangle}{\sqrt{\left\langle Y, Y\right\rangle}}  -4 \left\langle X, Y \right\rangle \left\langle X, U \right\rangle \left\langle V, Y \right\rangle \bigg]. 
	\end{split}
	\end{equation}
	From above equation, we can write
	\begin{equation}{\label{geoinfeq3}}
	\begin{split}
	g_{_{Y_{\mathfrak{m}}}}(Y_{\mathfrak{m}}, \left[ Y, Z\right]_{\mathfrak{m}} )&= \dfrac{\left\langle X, Y_{\mathfrak{m}}\right\rangle^2}{\left( \left\langle  X, Y_{\mathfrak{m}}\right\rangle - \sqrt{ \left\langle Y_{\mathfrak{m}} , Y_{\mathfrak{m}}\right\rangle } \right) ^4 } \bigg[ \left\langle X, Y_{\mathfrak{m}}\right\rangle^2 \left\langle X, \left[ Y, Z\right]_{\mathfrak{m}}\right\rangle \left\langle X, Y_{\mathfrak{m}} \right\rangle \\
	& \ \ \ -4 \left\langle Y_{\mathfrak{m}}, Y_{\mathfrak{m}}\right\rangle^ {3/2} \left\langle X, \left[ Y, Z\right]_{\mathfrak{m}} \right\rangle \left\langle X, Y_{\mathfrak{m}}\right\rangle 
	+ 6\left\langle Y_{\mathfrak{m}}, Y_{\mathfrak{m}}\right\rangle\left\langle X, \left[ Y, Z\right]_{\mathfrak{m}}\right\rangle \left\langle X, Y_{\mathfrak{m}}\right\rangle\\
	& \ \ \  + \dfrac{\left\langle X, Y_{\mathfrak{m}} \right\rangle^2 \left\langle X, \left[ Y, Z\right]_{\mathfrak{m}} \right\rangle \left\langle Y_{\mathfrak{m}}, Y_{\mathfrak{m}}\right\rangle }{\sqrt{\left\langle Y_{\mathfrak{m}}, Y_{\mathfrak{m}}\right\rangle}} -4 \left\langle X, Y_{\mathfrak{m}} \right\rangle  \left\langle X, \left[ Y, Z\right]_{\mathfrak{m}}\right\rangle \left\langle Y_{\mathfrak{m}}, Y_{\mathfrak{m}}\right\rangle   \\
	& \ \ \  - \dfrac{\left\langle X, Y_{\mathfrak{m}}\right\rangle^3 \left\langle Y_{\mathfrak{m}}, Y_{\mathfrak{m}} \right\rangle \left\langle \left[ Y, Z\right]_{\mathfrak{m}}, Y_{\mathfrak{m}}\right\rangle}{\left\langle Y_{\mathfrak{m}}, Y_{\mathfrak{m}}\right\rangle ^{3/2}} + \dfrac{\left\langle X, Y_{\mathfrak{m}}\right\rangle ^3 \left\langle Y_{\mathfrak{m}}, \left[ Y, Z\right]_{\mathfrak{m}}\right\rangle }{\sqrt{\left\langle Y_{\mathfrak{m}}, Y_{\mathfrak{m}}\right\rangle}}\\
	& \ \ \  +  \dfrac{ 4 \left\langle X, Y_{\mathfrak{m}}\right\rangle ^2 \left\langle Y_{\mathfrak{m}}, Y_{\mathfrak{m}}\right\rangle  \left\langle \left[ Y, Z\right]_{\mathfrak{m}}, Y_{\mathfrak{m}}\right\rangle  }{\left\langle Y_{\mathfrak{m}}, Y_{\mathfrak{m}}\right\rangle }  - \left\langle X,Y_{\mathfrak{m}} \right\rangle^2 \left\langle Y_{\mathfrak{m}}, \left[ Y, Z\right]_{\mathfrak{m}}\right\rangle \\
	& \ \ \ + \dfrac{\left\langle X, Y_{\mathfrak{m}} \right\rangle^2 \left\langle X, Y_{\mathfrak{m}}\right\rangle \left\langle \left[ Y, Z\right]_{\mathfrak{m}}, Y_{\mathfrak{m}} \right\rangle}{\sqrt{\left\langle Y_{\mathfrak{m}}, Y_{\mathfrak{m}}\right\rangle}}  -4 \left\langle X, Y_{\mathfrak{m}} \right\rangle \left\langle X, Y_{\mathfrak{m}} \right\rangle \left\langle \left[ Y, Z\right]_{\mathfrak{m}}, Y_{\mathfrak{m}} \right\rangle \bigg]. \\
	&= \dfrac{\left\langle X, Y_{\mathfrak{m}}\right\rangle^3}{\left( \left\langle  X, Y_{\mathfrak{m}}\right\rangle - \sqrt{ \left\langle Y_{\mathfrak{m}} , Y_{\mathfrak{m}}\right\rangle } \right) ^4 } \times \\
	&\ \ \ \ \bigg[ \left\langle X, \left[ Y, Z\right]_{\mathfrak{m}}\right\rangle \Big\{\left\langle X, Y_{\mathfrak{m}}\right\rangle^2 -4 \left\langle Y_{\mathfrak{m}}, Y_{\mathfrak{m}}\right\rangle ^{3/2} + \left\langle X, Y_{\mathfrak{m}}\right\rangle \sqrt{\left\langle Y_{\mathfrak{m}}, Y_{\mathfrak{m}}\right\rangle} + 2 \left\langle Y_{\mathfrak{m}}, Y_{\mathfrak{m}} \right\rangle    \Big\}\\
	& \ \ \ \  +\left\langle Y_{\mathfrak{m}}, \left[ Y, Z\right]_{\mathfrak{m}}\right\rangle \Big\{ \dfrac{\left\langle X, Y_{\mathfrak{m}}\right\rangle^2 }{\sqrt{\left\langle Y_{\mathfrak{m}}, Y_{\mathfrak{m}} \right\rangle }} - \left\langle X, Y_{\mathfrak{m}}\right\rangle \Big\}\bigg].
	\end{split}
	\end{equation}
	Now, from lemma (\ref{Latifigeolemma}),  $Y \in \mathfrak{g}$ is a geodesic vector if and only if $$ g_{_{Y_{\mathfrak{m}}}}(Y_{\mathfrak{m}}, \left[ Y, Z\right]_{\mathfrak{m}})=0, \ \forall \ Z \in \mathfrak{m}.$$
	Therefore $g_{_{Y_{\mathfrak{m}}}}(Y_{\mathfrak{m}}, \left[ Y, Z\right]_{\mathfrak{m}})=0$ if and only if equation (\ref{geoinfeq2}) holds. 
\end{proof}

\begin{corollary}
 Let $(G/H, F)$ be a homogeneous Finsler space with infinite series metric $F=\dfrac{\beta^2}{\beta-\alpha}$ defined by an invariant Riemannian metric $\left\langle \ , \ \right\rangle $ and an invariant vector field $\tilde{X}$ such that $\tilde{X}(H)= X.$ Let $ Y \in \mathfrak{g} $ be a vector such that $\left\langle X, \left[ Y, Z\right]_{\mathfrak{m}} \right\rangle = 0 \ \forall \ Z \in \mathfrak{m}  $. Then $Y$ is a geodesic vector of $\left( G/H, \left\langle \ , \ \right\rangle \right)  $ if and only if $Y$ is a geodesic vector of $\left( G/H, F\right) $.
\end{corollary}
\begin{proof}
From equation (\ref{geoinfeq3}), we can write
\begin{equation*}
\begin{split}
g_{_{Y_{\mathfrak{m}}}}(Y_{\mathfrak{m}}, \left[ Y, Z\right]_{\mathfrak{m}} )
&= \dfrac{\left\langle X, Y_{\mathfrak{m}}\right\rangle^3}{\left( \left\langle  X, Y_{\mathfrak{m}}\right\rangle - \sqrt{ \left\langle Y_{\mathfrak{m}} , Y_{\mathfrak{m}}\right\rangle } \right) ^4 } \times \\
&\ \ \ \ \bigg[ \left\langle X, \left[ Y, Z\right]_{\mathfrak{m}}\right\rangle \Big\{\left\langle X, Y_{\mathfrak{m}}\right\rangle^2 -4 \left\langle Y_{\mathfrak{m}}, Y_{\mathfrak{m}}\right\rangle ^{3/2} + \left\langle X, Y_{\mathfrak{m}}\right\rangle \sqrt{\left\langle Y_{\mathfrak{m}}, Y_{\mathfrak{m}}\right\rangle} + 2 \left\langle Y_{\mathfrak{m}}, Y_{\mathfrak{m}} \right\rangle   \Big\}\\
& \ \ \ \  +\left\langle Y_{\mathfrak{m}}, \left[ Y, Z\right]_{\mathfrak{m}}\right\rangle \Big\{ \dfrac{\left\langle X, Y_{\mathfrak{m}}\right\rangle^2 }{\sqrt{\left\langle Y_{\mathfrak{m}}, Y_{\mathfrak{m}} \right\rangle }} - \left\langle X, Y_{\mathfrak{m}}\right\rangle \Big\}\bigg]\\
&= \dfrac{\left\langle X, Y_{\mathfrak{m}}\right\rangle^4 \left\langle Y_{\mathfrak{m}}, \left[ Y, Z\right]_{\mathfrak{m}}\right\rangle}{\left( \left\langle  X, Y_{\mathfrak{m}}\right\rangle - \sqrt{ \left\langle Y_{\mathfrak{m}} , Y_{\mathfrak{m}}\right\rangle } \right) ^3 \sqrt{\left\langle Y_{\mathfrak{m}}, Y_{\mathfrak{m}} \right\rangle } }, \ \ \text{because} \ \  \left\langle X, \left[ Y, Z\right]_{\mathfrak{m}} \right\rangle = 0.
\end{split}
\end{equation*}
Therefore $g_{_{Y_{\mathfrak{m}}}}(Y_{\mathfrak{m}}, \left[ Y, Z\right]_{\mathfrak{m}})=0$ if and only if $\left\langle Y_{\mathfrak{m}}, \left[ Y, Z\right]_{\mathfrak{m}}\right\rangle = 0. $
\end{proof}

\begin{theorem}
	Let $(G/H, F)$ be a homogeneous Finsler space with infinite series metric $F=\dfrac{\beta^2}{\beta-\alpha}$ defined by an invariant Riemannian metric $\left\langle \ , \ \right\rangle $ and an invariant vector field $\tilde{X}$ such that $\tilde{X}(H)= X.$ Then $X$ is a geodesic vector of $(G/H, \left\langle \ , \ \right\rangle ) $ if and only if $X$ is a geodesic vector of $(G/H, F)$.
\end{theorem}
\begin{proof}
	
	From equation (\ref{geoinfeq1}), we can write 
	\begin{equation*}
	\begin{split}
	g_{_{X}}(X, \left[ X, Z\right]_{\mathfrak{m}}) &=  \dfrac{\left\langle X, X\right\rangle^2}{\left( \left\langle  X, X\right\rangle - \sqrt{ \left\langle X , X\right\rangle } \right) ^4 } 
	\bigg[ \left\langle X, X\right\rangle^2 \left\langle X, \left[ X, Z\right]_{\mathfrak{m}}\right\rangle \left\langle X, X \right\rangle \\
	& \ \ \ -4 \left\langle X, X\right\rangle^ {3/2} \left\langle X,\left[ X, Z\right]_{\mathfrak{m}} \right\rangle \left\langle X, X\right\rangle 
	+ 6\left\langle X, X\right\rangle\left\langle X,\left[ X, Z\right]_{\mathfrak{m}}\right\rangle \left\langle X, X\right\rangle \\
	& \ \ \ + \dfrac{\left\langle X, X \right\rangle^2 \left\langle X, \left[ X, Z\right]_{\mathfrak{m}} \right\rangle \left\langle X, X\right\rangle }{\sqrt{\left\langle X, X\right\rangle}} -4 \left\langle X, X \right\rangle  \left\langle X, \left[ X, Z\right]_{\mathfrak{m}}\right\rangle \left\langle X, X\right\rangle   \\
	&\ \ \   - \dfrac{\left\langle X, X\right\rangle^3 \left\langle X, X \right\rangle \left\langle\left[ X, Z\right]_{\mathfrak{m}}, X\right\rangle}{\left\langle X, X\right\rangle ^{3/2}} + \dfrac{\left\langle X, X\right\rangle ^3 \left\langle X, \left[ X, Z\right]_{\mathfrak{m}}\right\rangle }{\sqrt{\left\langle X, X\right\rangle}}\\ 
	&\ \ \  +  \dfrac{ 4 \left\langle X, X\right\rangle ^2 \left\langle X, X\right\rangle  \left\langle \left[ X, Z\right]_{\mathfrak{m}}, X\right\rangle  }{\left\langle X, X\right\rangle } 
	- \left\langle X, X \right\rangle^2 \left\langle X, \left[ X, Z\right]_{\mathfrak{m}}\right\rangle \\
	&\ \ \ + \dfrac{\left\langle X, X \right\rangle^2 \left\langle X, X\right\rangle \left\langle \left[ X, Z\right]_{\mathfrak{m}}, X \right\rangle}{\sqrt{\left\langle X, X\right\rangle}}  -4 \left\langle X, X \right\rangle \left\langle X, X \right\rangle \left\langle \left[ X, Z\right]_{\mathfrak{m}}, X \right\rangle \bigg]. \\
	&= \dfrac{\left\langle X, X\right\rangle^3 \left( \left\langle X, X\right\rangle ^2+ \left\langle X, X\right\rangle -2 \left\langle X, X\right\rangle \sqrt{\left\langle X, X \right\rangle }  \right) \left\langle X, \left[ X, Z\right]_{\mathfrak{m}}\right\rangle }{\left( \left\langle  X, X\right\rangle - \sqrt{ \left\langle X , X\right\rangle } \right) ^4 } \\
	&= \dfrac{\left\langle X, X\right\rangle^3  \left\langle X, \left[ X, Z\right]_{\mathfrak{m}}\right\rangle }{\left( \left\langle  X, X\right\rangle - \sqrt{ \left\langle X , X\right\rangle } \right) ^2 }.
	\end{split}
	\end{equation*}
	Therefore $g_{_{X}}(X, \left[ X, Z\right]_{\mathfrak{m}})=0$ if and only if $\left\langle X, \left[ X, Z\right]_{\mathfrak{m}}\right\rangle = 0. $
\end{proof}

Now, we give  necessary and sufficient condition for a nonzero vector in a homogeneous Finsler  space with  exponential metric to be a geodesic vector.

\begin{theorem}
	Let $G/H$ be a homogeneous Finsler space with exponential metric $F=\alpha  e^{\beta/\alpha}$ given by an invariant Riemannian metric $\tilde{a}$ and an invariant vector field $\tilde{X}$ such that $\tilde{X}(H)= X.$ Then, a non-zero vector $Y \in \mathfrak{g}$ is a geodesic vector if and only if 
	\begin{equation}{\label{geoexpeq2}}
	\begin{split}
	\left\langle X + \left( \dfrac{\sqrt{\left\langle Y_{\mathfrak{m}}, Y_{\mathfrak{m}}\right\rangle }- \left\langle X, Y_{\mathfrak{m}}\right\rangle}{\left\langle Y_{\mathfrak{m}}, Y_{\mathfrak{m}}\right\rangle }\right) Y_{\mathfrak{m}}, \left[Y, Z \right]_{\mathfrak{m}}  \right\rangle =0.
	\end{split}
	\end{equation}
\end{theorem}
\begin{proof}
	Using lemma 3.4 of (\cite{GKscur}), we can write 
	$$	F\left( Y\right) = \left\langle  X, Y\right\rangle e^{\left\langle  X, Y\right\rangle / \sqrt{ \left\langle Y , Y\right\rangle }}. $$
	Also, we know that 
	$$g_{_{Y}}(U, V)= \dfrac{1}{2}\dfrac{\partial^2}{\partial s \partial t}F^2 (Y+sU + tV)\bigg|_{s=t=0}.$$
	After  some calculations, we get
	\begin{equation}{\label{geoexpeq1}}
	\begin{split}
	g_{_{Y}}(U, V)= e^{2\left\langle  X, Y\right\rangle / \sqrt{ \left\langle Y , Y\right\rangle }}  \Bigg[ & \left\langle U, V\right\rangle + 2 \left\langle X, U \right\rangle  \left\langle X, V\right\rangle - \dfrac{\left\langle X, Y\right\rangle \left\langle Y, U\right\rangle \left\langle Y, V\right\rangle   }{\left\langle Y, Y\right\rangle^{3/2} } \\
	&   + \dfrac{1}{\sqrt{ \left\langle Y , Y\right\rangle }}\bigg\{ \left\langle X, U \right\rangle \left\langle Y, V\right\rangle + \left\langle X, V \right\rangle \left\langle Y, U\right\rangle - \left\langle X, Y\right\rangle \left\langle U, V \right\rangle \bigg\} \\
	&  + \dfrac{2\left\langle X, Y\right\rangle }{\left\langle Y, Y \right\rangle } \left\lbrace \dfrac{\left\langle X, Y\right\rangle \left\langle Y, U \right\rangle \left\langle Y, V\right\rangle}{\left\langle Y, Y\right\rangle } - \left\langle Y, U\right\rangle \left\langle X, V\right\rangle - \left\langle X, U \right\rangle \left\langle Y, V\right\rangle  \right\rbrace \Bigg].
	\end{split}
	\end{equation}
	From above equation, we can write
	\begin{equation}{\label{geoexpeq3}}
	\begin{split}
	&g_{_{Y_{\mathfrak{m}}}}(Y_{\mathfrak{m}}, \left[ Y, Z\right]_{\mathfrak{m}} ) = e^{2\left\langle  X, Y_{\mathfrak{m}}\right\rangle / \sqrt{ \left\langle Y_{\mathfrak{m}} , Y_{\mathfrak{m}}\right\rangle }} \times \\
	& \Bigg[  \left\langle Y_{\mathfrak{m}}, \left[ Y, Z\right]_{\mathfrak{m}}\right\rangle + 2 \left\langle X, Y_{\mathfrak{m}} \right\rangle  \left\langle X, \left[ Y, Z\right]_{\mathfrak{m}}\right\rangle - \dfrac{\left\langle X, Y_{\mathfrak{m}}\right\rangle \left\langle Y_{\mathfrak{m}}, Y_{\mathfrak{m}}\right\rangle \left\langle Y_{\mathfrak{m}}, \left[ Y, Z\right]_{\mathfrak{m}}\right\rangle   }{\left\langle Y_{\mathfrak{m}}, Y_{\mathfrak{m}}\right\rangle^{3/2} } \\
	&   + \dfrac{1}{\sqrt{ \left\langle Y_{\mathfrak{m}} , Y_{\mathfrak{m}}\right\rangle }}\bigg\{ \left\langle X, Y_{\mathfrak{m}} \right\rangle \left\langle Y_{\mathfrak{m}}, \left[ Y, Z\right]_{\mathfrak{m}}\right\rangle + \left\langle X, \left[ Y, Z\right]_{\mathfrak{m}} \right\rangle \left\langle Y_{\mathfrak{m}}, Y_{\mathfrak{m}}\right\rangle - \left\langle X, Y_{\mathfrak{m}}\right\rangle \left\langle Y_{\mathfrak{m}}, \left[ Y, Z\right]_{\mathfrak{m}} \right\rangle \bigg\} \\
	&  + \dfrac{2\left\langle X, Y_{\mathfrak{m}}\right\rangle }{\left\langle Y_{\mathfrak{m}}, Y_{\mathfrak{m}} \right\rangle } \left\lbrace \dfrac{\left\langle X, Y_{\mathfrak{m}}\right\rangle \left\langle Y_{\mathfrak{m}}, Y_{\mathfrak{m}} \right\rangle \left\langle Y_{\mathfrak{m}}, \left[ Y, Z\right]_{\mathfrak{m}}\right\rangle}{\left\langle Y_{\mathfrak{m}}, Y_{\mathfrak{m}}\right\rangle } - \left\langle Y_{\mathfrak{m}}, Y_{\mathfrak{m}}\right\rangle \left\langle X, \left[ Y, Z\right]_{\mathfrak{m}}\right\rangle - \left\langle X, Y_{\mathfrak{m}}\right\rangle \left\langle Y_{\mathfrak{m}}, \left[ Y, Z\right]_{\mathfrak{m}}\right\rangle  \right\rbrace \Bigg]\\
	&=    e^{2\left\langle  X, Y_{\mathfrak{m}}\right\rangle / \sqrt{ \left\langle Y_{\mathfrak{m}} , Y_{\mathfrak{m}}\right\rangle }}  \Bigg[ \left\langle Y_{\mathfrak{m}}, \left[ Y, Z\right]_{\mathfrak{m}}\right\rangle \left\lbrace 1 - \dfrac{\left\langle X, Y_{\mathfrak{m}}\right\rangle }{\sqrt{\left\langle Y_{\mathfrak{m}}, Y_{\mathfrak{m}} \right\rangle }}\right\rbrace + \left\langle X, \left[ Y, Z\right]_{\mathfrak{m}} \right\rangle \sqrt{\left\langle Y_{\mathfrak{m}}, Y_{\mathfrak{m}}\right\rangle } \Bigg]\\
	&=  \sqrt{ \left\langle Y_{\mathfrak{m}} , Y_{\mathfrak{m}}\right\rangle } \  e^{2\left\langle  X, Y_{\mathfrak{m}}\right\rangle / \sqrt{ \left\langle Y_{\mathfrak{m}} , Y_{\mathfrak{m}}\right\rangle }} \left\langle X + \left( \dfrac{\sqrt{\left\langle Y_{\mathfrak{m}}, Y_{\mathfrak{m}}\right\rangle }- \left\langle X, Y_{\mathfrak{m}}\right\rangle}{\left\langle Y_{\mathfrak{m}}, Y_{\mathfrak{m}}\right\rangle }\right) Y_{\mathfrak{m}}, \left[Y, Z \right]_{\mathfrak{m}}  \right\rangle.
	\end{split}
	\end{equation}
	Now, from lemma (\ref{Latifigeolemma}),  $Y \in \mathfrak{g}$ is a geodesic vector if and only if $$ g_{_{Y_{\mathfrak{m}}}}(Y_{\mathfrak{m}}, \left[ Y, Z\right]_{\mathfrak{m}})=0, \ \forall \ Z \in \mathfrak{m}.$$
	Therefore $g_{_{Y_{\mathfrak{m}}}}(Y_{\mathfrak{m}}, \left[ Y, Z\right]_{\mathfrak{m}})=0$ if and only if equation (\ref{geoexpeq2}) holds. 
\end{proof}

\begin{corollary}
	Let $(G/H, F)$ be a homogeneous Finsler space with exponential metric $F= \alpha  e^{\beta/\alpha}$ defined by an invariant Riemannian metric $\left\langle \ , \ \right\rangle $ and an invariant vector field $\tilde{X}$ such that $\tilde{X}(H)= X.$ Let $ Y \in \mathfrak{g} $ be a vector such that $\left\langle X, \left[ Y, Z\right]_{\mathfrak{m}} \right\rangle = 0 \ \forall \ Z \in \mathfrak{m}  $. Then $Y$ is a geodesic vector of $\left( G/H, \left\langle \ , \ \right\rangle \right)  $ if and only if $Y$ is a geodesic vector of $\left( G/H, F\right) $.
\end{corollary}
\begin{proof}
	From equation (\ref{geoexpeq3}), we can write
	\begin{equation*}
	\begin{split}
	g_{_{Y_{\mathfrak{m}}}}(Y_{\mathfrak{m}}, \left[ Y, Z\right]_{\mathfrak{m}} )
	 = &   e^{2\left\langle  X, Y_{\mathfrak{m}}\right\rangle / \sqrt{ \left\langle Y_{\mathfrak{m}} , Y_{\mathfrak{m}}\right\rangle }}  \Bigg[ \left\langle Y_{\mathfrak{m}}, \left[ Y, Z\right]_{\mathfrak{m}}\right\rangle \left\lbrace 1 - \dfrac{\left\langle X, Y_{\mathfrak{m}}\right\rangle }{\sqrt{\left\langle Y_{\mathfrak{m}}, Y_{\mathfrak{m}} \right\rangle }}\right\rbrace + \left\langle X, \left[ Y, Z\right]_{\mathfrak{m}} \right\rangle \sqrt{\left\langle Y_{\mathfrak{m}}, Y_{\mathfrak{m}}\right\rangle } \Bigg]\\
	 = & e^{2\left\langle  X, Y_{\mathfrak{m}}\right\rangle / \sqrt{ \left\langle Y_{\mathfrak{m}} , Y_{\mathfrak{m}}\right\rangle }}   \left\lbrace 1 - \dfrac{\left\langle X, Y_{\mathfrak{m}}\right\rangle }{\sqrt{\left\langle Y_{\mathfrak{m}}, Y_{\mathfrak{m}} \right\rangle }}\right\rbrace  \left\langle Y_{\mathfrak{m}}, \left[ Y, Z\right]_{\mathfrak{m}}\right\rangle , \ \ \text{because} \ \  \left\langle X, \left[ Y, Z\right]_{\mathfrak{m}} \right\rangle = 0.
	\end{split}
	\end{equation*}
	Therefore $g_{_{Y_{\mathfrak{m}}}}(Y_{\mathfrak{m}}, \left[ Y, Z\right]_{\mathfrak{m}})=0$ if and only if $\left\langle Y_{\mathfrak{m}}, \left[ Y, Z\right]_{\mathfrak{m}}\right\rangle = 0. $
\end{proof}
\begin{theorem}
	Let $(G/H, F)$ be a homogeneous Finsler space with exponential metric $F=\alpha  e^{\beta/\alpha}$ defined by an invariant Riemannian metric $\left\langle \ , \ \right\rangle $ and an invariant vector field $\tilde{X}$ such that $\tilde{X}(H)= X.$ Then $X$ is a geodesic vector of $(G/H, \left\langle \ , \ \right\rangle ) $ if and only if $X$ is a geodesic vector of $(G/H, F)$.
\end{theorem}
\begin{proof}
	
	From equation (\ref{geoexpeq1}), we can write 
	\begin{equation*}
	\begin{split}
	g_{_{X}}(X, \left[ X, Z\right]_{\mathfrak{m}}) &=  e^{2\left\langle  X, X\right\rangle / \sqrt{ \left\langle X , X\right\rangle }} \times \\
	&  \Bigg[  \left\langle X, \left[ X, Z\right]_{\mathfrak{m}}\right\rangle + 2 \left\langle X, X \right\rangle  \left\langle X, \left[ X, Z\right]_{\mathfrak{m}}\right\rangle - \dfrac{\left\langle X, X\right\rangle \left\langle X, X\right\rangle \left\langle X, \left[ X, Z\right]_{\mathfrak{m}}\right\rangle   }{\left\langle X, X\right\rangle^{3/2} } \\
	&   + \dfrac{1}{\sqrt{ \left\langle X , X\right\rangle }}\bigg\{ \left\langle X, X \right\rangle \left\langle X, \left[ X, Z\right]_{\mathfrak{m}}\right\rangle + \left\langle X, \left[ X, Z\right]_{\mathfrak{m}} \right\rangle \left\langle X, X\right\rangle - \left\langle X, X\right\rangle \left\langle X, \left[ X, Z\right]_{\mathfrak{m}} \right\rangle \bigg\} \\
	&  + \dfrac{2\left\langle X, X\right\rangle }{\left\langle X, X \right\rangle } \left\lbrace \dfrac{\left\langle X, X\right\rangle \left\langle X, X \right\rangle \left\langle X, \left[ X, Z\right]_{\mathfrak{m}}\right\rangle}{\left\langle X, X\right\rangle } - \left\langle X, X\right\rangle \left\langle X, \left[ X, Z\right]_{\mathfrak{m}}\right\rangle - \left\langle X, X \right\rangle \left\langle X, \left[ X, Z\right]_{\mathfrak{m}}\right\rangle  \right\rbrace \Bigg]\\
	&= e^{2\left\langle  X, X\right\rangle / \sqrt{ \left\langle X , X\right\rangle }}   \left\langle X, \left[ X, Z\right]_{\mathfrak{m}}\right\rangle.
	\end{split}
	\end{equation*}
	Therefore $g_{_{X}}(X, \left[ X, Z\right]_{\mathfrak{m}})=0$ if and only if $\left\langle X, \left[ X, Z\right]_{\mathfrak{m}}\right\rangle = 0. $
\end{proof}

\end{document}